\documentclass[10pt]{amsart}
\usepackage{amsmath,amssymb,amsthm,graphicx}

\usepackage{pgf}
\usepackage{tikz}
\usepackage{tikz-cd}

\usetikzlibrary{matrix,arrows,backgrounds,shapes.misc,shapes.geometric,patterns,calc,positioning}
\usetikzlibrary{calc,shapes}
\usetikzlibrary{decorations.pathmorphing}

\usepackage{lmodern}
\usepackage[T1]{fontenc}

\usepackage{wrapfig}
\usepackage{epsfig}  

\usepackage{amsfonts}

\usepackage{multicol}
\usepackage[all]{xy}
\usepackage{xcolor}
\usepackage{color}
\usepackage{float}

\usepackage{graphics}

\usepackage[left=3cm,right=3cm,top=3cm,bottom=3cm]{geometry}

\newcommand{\chara}{\mathrm{char}\hspace{1pt}}
\newcommand{\ttop}{\mathrm{top}\hspace{1pt}}

\newcommand{\End}{{\textup{End}}}
\newcommand{\Hom}{\textup{Hom}}
\newcommand{\Ext}{{\textup{Ext}}}

\newcommand{\rad}{{\textup{rad}}}

\newcommand{\proj}{{\sf proj \hspace{.02in} }}

\newcommand{\add}{{\textup{add}\hspace{1pt}}}

\newcommand{\injdim}{{\textup{inj.dim}\hspace{1pt}}}
\newcommand{\projdim}{{\textup{proj.dim}\hspace{1pt}}}

\newcommand{\Aa}{\mathbb{A}}
\newcommand{\tAa}{\tilde{\mathbb{A}}}

\newcommand{\Cc}{\mathcal{C}}

\newcommand{\Db}{\mathcal{D}^b}

\newcommand{\Tt}{\mathcal{T}}
\newcommand{\Z}{\mathbb{Z}}

\newcommand{\xra}{\xrightarrow}
\newcommand{\ra}{\rightarrow}

\newtheorem{teo}{Theorem}[section]
\newtheorem{lema}[teo]{Lemma}
\newtheorem{teorema}[teo]{Theorem}

\newtheorem{corolario}[teo]{Corollary}

\newtheorem{proposicion}[teo]{Proposition}

\newtheorem{ejemplo}[teo]{Example}
\newtheorem{remark}[teo]{Remark}
\newtheorem{question}[teo]{Question}

\newtheorem{definicion}[teo]{Definition}

\newtheorem{teorema*}{Theorem}
\newtheorem{corolario*}{Corollary}

\usepackage{color}

\usepackage{pdfpages}

\DeclareMathOperator{\GP}{GP}
\DeclareMathOperator{\fgmod}{mod}
\DeclareMathOperator{\stabGP}{\underline{GP}}
\DeclareMathOperator{\stabfgmod}{\underline{mod}}

\DeclareMathOperator{\stabHom}{\underline{Hom}}

\newcommand{\id}{\mathrm{Id}}
\newcommand{\iso}{\cong}

\newcommand{\kdual}{\mathrm{D}}

\theoremstyle{plain}
\newtheorem{thm}{Theorem}[section]
\newtheorem{prop}[thm]{Proposition}
\newtheorem{cor}[thm]{Corollary}
\newtheorem{lem}[thm]{Lemma}
\theoremstyle{remark}

\begin{document}
\date{\today. Key words: 2-Calabi-Yau tilted algebras, monomial algebras, Jacobian algebras, Gorenstein-projective modules. 2010 Mathematics Subject Classification: 16G20, 16E65, 18E30.}

\title{Monomial Gorenstein algebras and the stably Calabi--Yau property}
\author{Ana Garcia Elsener}

\begin{abstract} 

A celebrated result by Keller--Reiten says that $2$-Calabi--Yau tilted algebras are Gorenstein and stably 3-Calabi--Yau. This note shows that the converse holds in the monomial case: a 1-Gorenstein monomial algebra and stably 3-Calabi--Yau is 2-Calabi--Yau tilted, moreover is Jacobian. We study the case of other stably Calabi--Yau Gorenstein monomial algebras.

\end{abstract}

\maketitle

\section{Introduction}\label{sect 1}

Cluster-tilted algebras are inspired by Fomin--Zelevinsky cluster algebras, a class of commutative algebras that was introduced in \cite{FZ}, also they are related to tilted algebras \cite{HR}. Whereas  tilted algebras are defined as endomorphism algebras of tilting objects over an hereditary algebra, a cluster-tilted algebra is the endomorphism algebra of a cluster-tilting object in a triangulated category called the cluster category of a hereditary algebra. The crucial step to define a cluster algebra, and the way to obtain new cluster-tilted algebras from old is called mutation. The mutation algorithm takes a quiver (or matrix) and transforms it into a new one, and the cluster-tilting mutation process takes a certain finite dimensional  algebra and defines a new one that is closely related to the original. All the algebras given by this operation share interesting homological properties and their module categories are very close to each other. Cluster categories and cluster-tilted algebras were introduced in \cite{BMRRT,CCS,BMR}.

The clear interaction between cluster algebras and cluster tilted algebras has motivated hundreds of works. Some of them are devoted to study other 'cluster like' categories and their cluster-tilting objects. So is the case of 2-Calabi--Yau tilted algebras which are obtained by replacing the cluster category by a 2-Calabi--Yau triangulated category. The definition of 2-Calabi--Yau tilted algebra appeared explicitly in \cite{BIKR}. The reader can also find important references in \cite{R}. One of the main articles on the theory is the one by Keller--Reiten \cite{KR} where it is proved that 2-Calabi--Yau tilted algebras are 1-Iwanaga-Gorenstein and stably 3-Calabi--Yau, i. e. with a 3-Calabi--Yau stable category of Gorenstein projective objects (singularity category).

A large family of 2-Calabi--Yau tilted algebras are Jacobian algebras of a quiver with potential $(Q,W)$. Jacobian algebras are defined by applying cyclic derivatives to a quiver with potential (a generalization called quiver with hyperpotential can be considered \cite{Lad}). As before, new quivers with potential can be obtained from old ones by mutation, see \cite{DWZ}. The 2-Calabi--Yau category that allows us to obtain Jacobian algebras as endomorphism algebras of cluster-tilting objects is called the generalized cluster category $\Cc_{(Q,W)}$ and was introduced by Amiot \cite{Am} building on previous works by Keller and Ginzburg.

On the other hand, in recent years several authors studied Georenstein monomial algebras and their singularity categories \cite{LZ,CSZ}. These articles show that combinatorics behind Gorenstein projective stable categories is nice when we are dealing with 1-Iwanaga--Gorenstein monomial algebras.

Form the mentioned above, we know that Jacobian algebras are 2-Calabi--Yau tilted, and also that 2-Calabi--Yau tilted algebras are 1-Iwanaga-Gorenstein and stably 3-Calabi--Yau. But are these conditions equivalent? 

In the case of monomial algebras, this article uses the combinatorics of Gorenstein projective stable categories to show that the answer is affirmative.


The main result of the article is the following:

\begin{teorema*}[Theorem \ref{main-theorem-1}] Let $k$ be an algebraically closed field of characteristic zero and $\Lambda=kQ/I$ a monomial $k$-algebra. The following are equivalent: 
\begin{enumerate}
\item $\Lambda$ is $2$-Calabi--Yau tilted,
\item $\Lambda$ is a $1$-Iwanaga--Gorenstein algebra and stably $3$-Calabi--Yau,
\item $\Lambda$ is Jacobian.
\end{enumerate}
\end{teorema*}

With the consideration of hyperpotentials instead of potentials, the condition $\chara k = 0$ can be removed.

Jacobian algebras include a well known class of gentle 2-Calabi--Yau tilted algebras whose quivers with potential can be obtained from triangulations of unpunctured surfaces, see \cite{ABCP}. This family includes the cluster-tilted algebras of type $\Aa$ and $\tAa$. A related family of gentle algebras, that also have a geometric realization via unpunctured surfaces, are the $m$-cluster-tilted algebras of types $\Aa$ and $\tAa$, see \cite{T,Ba,Gub}. 
This family is a subfamily of gentle algebras defined by a bound quiver that can be realized as a dual quiver obtained from $(m+2)$-angulations as well. In Section 4 we consider a family of algebras arising from surfaces that generalize the class of $m$-cluster tilted algebras of types $\Aa$ and $\tAa$. The main result of this section characterizes their Gorenstein projective modules.  

\begin{teorema*}[Theorem \ref{main}] Let $\Lambda_\Tt = k Q_\Tt /I_\Tt$ be an algebra arising from a $(m+2)$-angulation and $N$ a $\Lambda_\Tt$-module. Then $N$ is Gorenstein projective if and only if $\Omega^{m+1}\tau N \simeq N$ over $\stabfgmod \Lambda_\Tt$. In particular, this holds for $m$-cluster tilted algebras of type $\Aa$ and $\tAa$.
\end{teorema*}

\section{Preliminaries}\label{sect 2}

Throughout these notes, we consider $k$-algebras where $k$ is an algebraically closed field of characteristic zero (if the contrary is not stated). Let $Q=(Q_0,Q_1)$ be a finite quiver, where $Q_0$ is the set of vertices and $Q_1$ the set of arrows. Let $s,t \colon Q_1 \ra Q_0$ be the functions that indicate the source and the target of each arrow, respectively. We will only consider \emph{finite dimensional basic $k$-algebras}. Every finite dimensional basic $k$-algebra $\Lambda$ is isomorphic to a quotient $kQ/I$, where $I$ is an admissible ideal. The pair $(Q,I)$ is called a \emph{bound quiver} \cite[Chapter III]{ASS}.

Let $\Omega$ be the usual syzygy operator, $\tau$ the Auslander--Reiten (AR) translation, and $D = \Hom_k ( - ,k)$. Denote by $\Db (\Aa_m)$ the derived category of a hereditary algebra of Dynkin type $\Aa_m$.

\begin{definicion}  \normalfont A $k$-algebra $\Lambda$ is \emph{Iwanaga--Gorenstein} if $\injdim \Lambda = \projdim D(\Lambda^{op}) \leq d$ for some non-negative integer $d$. In this case we say that $\Lambda$ is $d$-Iwanaga--Gorenstein.  

\end{definicion}

\begin{remark}\label{rema gorenstein} \normalfont For a $d$-Iwanaga--Gorenstein algebra $\Lambda$, a $\Lambda$-module $M$ is Gorenstein-projective if and only if $M$ is a $d$-th syzygy, see \cite[Proposition 6.20]{Be}. In this case each $\Lambda$-module either has infinite projective dimension or has projective dimension at most $d$. The stable subcategory of Gorenstein-projetive modules, also called singularity category, $\stabGP (\Lambda)$ is triangulated, with shift given by the formal inverse of $\Omega$.   
\end{remark}

An algebra $\Lambda= kQ/I$ is \emph{monomial} if all the generators of $I$ are paths $\alpha_1 \ldots \alpha_t$. Denote by $F$ the minimal set of paths generating $I$. Following \cite[Section 6.2]{LZ}, when $kQ/I$ is monomial and 1-Iwanaga--Gorenstein the elements of $F$ are very particular: $F$ is a disjoint union $\coprod_{i=1}^m F_i$ where the  subset $F_i$ consists of all sub-paths over a cyclic path $c_i = x_1 \xra{\alpha_1}  \cdots \ra x_{n{_i}}\xra{\alpha_{n_i}} x_1 $ of certain common length. The cycles $c_i$ are considered up to cyclic equivalence, and without arrow repetition. Denote by $\Cc(\Lambda)$ the set of cycles $c_1, \ldots, c_m$ that contain the paths in $F_1, \ldots, F_m$ respectively, by $n_i$ the length of the cycle $c_i$ and by $r_i$ the (common) length of paths in $F_i$. We refer to the generators in $F$ as \emph{zero-relations}. By \cite[Lemma 6.12]{LZ}, two different cycles $c_i,c_j$ in $\Cc(\Lambda)$ do not share arrows.

\begin{ejemplo}\label{ejemplo-1} Let $Q$ be the quiver with ideal $I= < \alpha \beta \gamma , \beta \gamma\delta, \gamma \delta \alpha, \delta \alpha \beta , \lambda^4 >$
\hspace{18pt} \begin{tikzcd}
1 \arrow[r, "\alpha"] 
& 2 \arrow[d, "\beta"] \\
4 \arrow[u,"\delta"]
& 3 \arrow[l,"\gamma"] \arrow[loop right,"\lambda"]
\end{tikzcd} 

Then $F = F_1 \coprod F_2$ with $F_1 = \{ \alpha \beta \gamma , \beta \gamma\delta, \gamma \delta \alpha, \delta \alpha \beta \}$, $F_2 = \{ \lambda^4 \}$. We have $c_1 = \alpha \beta \gamma \delta$ and $c_2 = \lambda$, so $n_1 = 4$, $n_2 = 1$, $r_1=3$ and $r_2=4$. 
\end{ejemplo}

\begin{lema}\label{lema LZ} \cite[Theorem 6.16]{LZ} Let $\Lambda=kQ/I$ be a $1$-Iwanaga--Gorenstein monomial algebra. The singularity category $\stabGP(\Lambda)$ is equivalent to the coproduct $\coprod_{c_i\in \Cc(\Lambda)} \Db(\Aa_{r_i -1}) / \tau^{n_i}$. 

\end{lema}

\subsection{Gentle algebras} We recall the definition of gentle algebra. By \cite{GR} it is known that gentle algebras are Iwanaga--Gorenstein.

\begin{definicion} \normalfont A $k$-algebra $\Lambda = k Q/I$  is \emph{gentle} if 
\begin{enumerate}
\item[(g1)] For each $x_0 \in Q_0$ there are at most two arrows with source $x_0$, and at most two arrows with target $x_0$,
\item[(g2)] the ideal $I$ is generated by paths  of length $2$, 
\item[(g3)] for each $\beta \in Q_1$ there is at most one arrow $\alpha \in Q_1$ and at most one arrow
$\gamma$ such that $\alpha \beta \in I$ and $\beta \gamma \in I$, and
\item[(g4)] for each $\beta \in Q_1$ there is at most one arrow $\alpha$ and at most one arrow
$\gamma$ such that $\alpha \beta \notin I$ and $\beta \gamma \notin I$.
\end{enumerate} 
\end{definicion}

An algebra $\Lambda = kQ/I$,
where $I$ is generated by paths and $(Q, I)$ satisfies the two conditions (g1)
and (g4), is called a \emph{string algebra}. Thus every gentle algebra is
a string algebra. A \emph{string} in $\Lambda$ is by definition a reduced walk $w$ in $Q$ avoiding the
zero-relations, thus $w$ is a sequence $x_1 \overset{\alpha_1}{\longleftrightarrow} x_2 \overset{\alpha_2}{\longleftrightarrow} \cdots \overset{\alpha_n}{\longleftrightarrow} x_{n+1}$ where the $x_i$ are vertices of $Q$ and each $\alpha_i$ is an arrow between the vertices $x_i$ and $x_{i+1}$ in either direction such that there is no $\overset{\beta}{\longrightarrow} \overset{\beta}{\longleftarrow}$, no $\overset{\beta}{\longleftarrow} \overset{\beta}{\longrightarrow} $, and no $\overset{\beta_1}{\longleftarrow} \cdots \overset{\beta_t}{\longleftarrow}$ or $ \overset{\beta_1}{\longrightarrow} \cdots \overset{\beta_t}{\longrightarrow}$ with $\beta_1 \ldots \beta_t \in I$. If the first and the last vertex of $w$ coincide, then the string is cyclic. A \emph{band} is a cyclic string $b$ such that each power $b^n$ is a cyclic string but $b$ is not a power of some string. The classification of indecomposable modules over a string
algebra $\Lambda = kQ/I$ is given by Butler--Ringel \cite{BuRi} in terms of strings and bands in $(Q,I)$. Each string $w$ defines an indecomposable module $M(w)$, called a \emph{string module}, and each band $b$ defines a family of indecomposable modules $M(b,\lambda,n)$, called \emph{band modules}, with parameters $\lambda \in k$ and $n \in \mathbb{N}$.

\begin{definicion}\label{gentle-def}  \normalfont Let $\Lambda = kQ/I$ be a gentle algebra.

\begin{enumerate}

\item[(a)] A cycle $x_1 \xra{\alpha_1}  \cdots \ra x_{n}\xra{\alpha_n} x_1 $ is \emph{saturated} if $\alpha_i \alpha_{i+1} \in I$, for $i \in \mathbb{Z}/ n \mathbb{Z}$. In particular, a \emph{saturated loop} is an arrow $\delta$ such that $s(\delta)=t(\delta)$ and $\delta^2 \in I$. 

\item[(b)] An arrow $\beta $ is \emph{gentle} if there is no arrow $\alpha$ such that $\alpha \beta \in I$.

\item[(c)] A path $\alpha_1 \ldots \alpha_n$ is formed by \emph{consecutive relations} if $\alpha_i \alpha_{i+1} \in I$ for $1 \leq i < n$.

\item[(d)] A path $\alpha_1 \ldots \alpha_n$ is \emph{critical} if it is formed by consecutive relations and $\alpha_1 $ is a gentle arrow.

\end{enumerate}

\end{definicion}

When there is no gentle arrow, we set $n(\Lambda)=0$. When there is a gentle arrow, let $n(\Lambda)$ be the maximal length computed over all critical paths. This number is bounded since $Q$ is finite.

\begin{teorema} \label{lema GR} \cite{GR} Let  $ \Lambda= kQ/I$ be a gentle algebra with $n(\Lambda)$ the maximum
length over all critical paths. Then $\injdim \Lambda = n(\Lambda) =
\projdim D(\Lambda^{op})$ if $n(\Lambda) > 0$, and $\injdim \Lambda = \projdim D(\Lambda^{op}) \leq 1$ if $n(\Lambda) = 0$. In particular, $\Lambda$ is Iwanaga--Gorenstein.

\end{teorema}

Let $x_1 \xra{\alpha_1} \cdots \xra{\alpha_{n-1}} x_n \xra{\alpha_n} x_1$ be a saturated cycle. Let $u_i$ be the maximal path avoiding zero-relations starting at the vertex $x_i$ and let  $v_i$ be the the maximal path avoiding zero-relations ending at $x_i$. Then the indecomposable projective module $P(x_i)$ and the indecomposable injective module $I(x_i)$ are string modules given by $P(x_i)=M(u_i^{-1} \alpha_i u_{i+1})$ and $I(x_i)= M(v_{i-1} \alpha_{i-1} v_{i}^{-1})$, see Figure \ref{kalck1}.

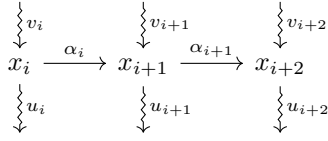
\begin{figure}[h!]

\begin{center}
\begin{tikzcd}[column sep=small]
\ \arrow[d,rightsquigarrow,"v_{i}"]& & \  \arrow[d,rightsquigarrow,"v_{i+1}"] & & \ \arrow[d,rightsquigarrow,"v_{i+2}"] \\ x_i \arrow[rr,"\alpha_i"] \arrow[d,rightsquigarrow,"u_i"] & & x_{i+1} \arrow[rr,"\alpha_{i+1}"] \arrow[d,rightsquigarrow,"u_{i+1}"]& & x_{i+2}\arrow[d,rightsquigarrow,"u_{i+2}"] \\
\ & & \ & & \
\end{tikzcd}
\end{center}
\caption{Local situation for a saturated cycle. }\label{kalck1}
\end{figure}

\begin{teorema}\cite[Theorem 2.5]{Ka}\label{teorema de kalck} Let $\Lambda = k Q/I$ be a gentle algebra. Let $x_1 \xra{\alpha_1} \cdots \xra{\alpha_{n-1}} x_n \xra{\alpha_n} x_1$ be a saturated cycle. The string module $M(u_i)$, where $u_i$ is the string starting at $x_i$, is Gorenstein-projective. Moreover, all indecomposable modules in $\stabGP(\Lambda)$ are obtained in such manner. 
\end{teorema}

\section{Monomial 2-Calabi--Yau tilted algebras are Jacobian}\label{sect-3}

A triangulated $k$-category $\Cc$, $\Hom$-finite with split idempotents, is  \emph{$d$-Calabi-Yau} if there is a bifunctorial isomorphism $\Hom_\Cc (X,Y)\simeq D\Hom_\Cc(Y,X[d])$ for  all $ X,Y\in   \Cc$. This means that the $d$-th power of $[1]$ is a Serre functor for the category, and this is to say that there is a funtorial isomorphism $ \tau^{-1}[d-1] = \mathrm{Id}$.

\begin{definicion} \normalfont Let $\Cc$ be a $2$-Calabi--Yau category \begin{enumerate}
\item An object $T$ is \emph{cluster-tilting} if it is basic and $\add T = \{ X\in \Cc \colon  \Hom_\Cc(X,T[1])=0 \}$. 
\item The endomorphism algebra $\End_\Cc(T)$ of a cluster-tilting object is called a \emph{$2$-Calabi--Yau tilted algebra}. 
\end{enumerate} 
\end{definicion}

Examples of $2$-Calabi--Yau tilted algebras are the \emph{cluster-tilted} algebras defined in \cite{BMR}. A more general family is obtained from generalized cluster categories defined by Amiot, see \cite{Am}. \emph{Quivers with potential} were introduced in \cite{DWZ}. A potential $W$ is a (possibly infinite) linear combination of cycles in
$Q$, up to cyclic equivalence. Given an arrow $\alpha$ and a cycle $c=\alpha_1 \ldots \alpha_l$, the cyclic derivative $\partial_\alpha (c)$ is defined by \[\partial_\alpha (\alpha_1 \ldots \alpha_l)= \sum_{k+1}^{l} \delta_{\alpha \alpha_k} \alpha_{k+1} \ldots \alpha_l \alpha_1 \ldots \alpha_{k-1}, \]
where $\delta_{\alpha \alpha_k}$ is the Kronecker delta. This derivative can be applied to a sum of cycles extending by linearity. Notice that a cycle $\alpha_1 \ldots \alpha_l$ may have arrow repetitions. Let $R \langle\langle Q\rangle\rangle$ be the complete path algebra consisting of all (possibly
infinite) linear combinations of paths in $Q$. Let $(Q,W)$ be a quiver with potential, the \emph{Jacobian algebra} is defined to be $Jac(Q,W)= R\langle\langle Q\rangle\rangle/ \overline{\langle\partial_\alpha W, \alpha \in Q_1\rangle}$. When one considers the algebra $R \langle \langle Q\rangle \rangle / \langle \partial_\alpha W, \alpha \in Q_1\rangle$ and the result is finite dimensional then it is a Jacobian algebra $Jac(Q,W)$ and there is no need to consider completions and ideal closures.

Amiot \cite[Sec. 3]{Am} showed that Jacobian algebras are $2$-CY tilted whenever they are finite dimensional by defining a generalized cluster category $\Cc_{(Q,W)}$. In \cite[Question 2.20]{Am2}, the author asked whether all $2$-Calabi--Yau tilted algebras are  Jacobian, and warned that the answer might be negative in a general setting.  The first part of this note is devoted to this question, provided the algebra is monomial.

Now we recall some main properties of 2-Calabi--Yau tilted algebras. The first part of the next theorem is a celebrated result by Keller--Reiten, the second part is a characterization of Gorenstein projective modules that can be computed over the module category.

\begin{teorema}\label{prop KR - teo ralf} \cite{KR,GS} Let $\Lambda$ be a $2$-Calabi--Yau tilted algebra, 
\begin{enumerate}
\item $\Lambda$ is $1$-Iwanaga--Gorenstein and stably $3$-Calabi--Yau. 
\item A $\Lambda$-module $M$ is Gorenstein projective if and only if $\Omega^2 \tau M \simeq M$ in $\stabfgmod \Lambda$.
\end{enumerate} 
\end{teorema}

Let $\Lambda = kQ/I$ be a monomial 2-Calabi--Yau tilted algebra. Since it is a monomial 1-Iwanaga--Gorenstein algebra by the result above, the singularity category $\stabGP(\Lambda)$ is described in Lemma \ref{lema LZ}. Using this, we can gather information on cycles and relations in $(Q,I)$. Remember that the zero-relations of length $r_i$ in $F$ lie in a cycle $c_i$ of length $n_i$.

\begin{lema} \label{lema 3-CY} Let $\Lambda = kQ/I$ be a monomial $1$-Iwanaga--Gorenstein algebra and stably $3$-Calabi--Yau, and $F= \coprod F_i$ the corresponding set of zero-relations. Then $r_i= b_i n_i - 1$, where $b_i \in \Z^+$ is such that $r_i > 0$.

\end{lema}

\begin{proof}
We want to know for which parameters $x,n \in\Z^+$ we have that an orbit category $ \Db(\Aa_x) / \tau^{n}$ is 3-Calabi--Yau. Recall that in the orbit category $[1]$ and $\tau$ are induced from $\Db (\Aa_x)$. 
So, at least, we need the functor $\tau^{-1} [2]$ to send objects of a certain orbit to objects in the same orbit. One can easily compute that $\tau^{-1}[2]X=\tau^{-x-2}X$ over $\Db (\Aa_x)$. Hence $\tau^{x+2} X$ has to be $(\tau^n)^b X = \mathrm{Id} X =X$ in the orbit category, so $x=n  b -2$. If we go back to the notation in Lemma \ref{lema LZ}, the claim follows.
\end{proof}

With the last lemma we have restricted the possible lengths of cycles and zero-relations in $kQ/I$, so now there is a constraint for the elements in $F$. In the following we show that a finite dimensional algebra described by a bound quiver $(Q,I)$ satisfying this constraint has to be 2-Calabi--Yau tilted.

\begin{lema}\label{lema-Jacobian} Let $\Lambda=kQ/I$ be a $1$-Iwanaga--Gorenstein monomial algebra where $F= \coprod_{i \in [1,m]} F_{i}$ is such that $r_i= b_i n_i - 1$, $b_i \in \Z^+$ and $r_i > 0$. Then the zero-relations arise from a potential $W= \sum_{i\in [1,m]}  c_i^{b_i}$.

\end{lema}

\begin{proof}
Since two different cycles in $\Cc(\Lambda)$ do not share arrows, each $\alpha \in Q_1$ is either in one of the cycles or is not present in $\Cc(\Lambda)$. If $\alpha$ is not in any cycle then the associated partial derivation adds zero to the ideal $I$. If $\alpha$ is in the cycle $c_j \in \Cc(\Lambda)$ of length $n_j$ (with no repetition), then $\partial_\alpha (\sum_i  c_i^{b_i})= \partial_{\alpha} (c_j^{b_j})$ and this is an element of $k Q$ given by $b_j u$  where $u$ is a subpath of length $r_j = n_j b_j - 1$ in the cycle $c_j^{b_j}$ without the arrow $\alpha$. Our general hypothesis is that $k$ is a field of characteristic zero, so these paths $b_j u$ can be replaced by $u$ as generators of $I$. Doing this for all $\alpha \in Q_1$ we cover exactly all the elements in $F$. Therefore $\Lambda$ is a Jacobian algebra hence it is 2-Calabi--Yau tilted.\end{proof}

The two lemmas above sum up in the following theorem.

\begin{teorema}\label{main-theorem-1} Let $k$ be an algebraically closed field of characteristic zero and $\Lambda=kQ/I$ a monomial $k$-algebra. The following are equivalent: 
\begin{enumerate}
\item $\Lambda$ is $2$-Calabi--Yau tilted,
\item $\Lambda$ is a $1$-Iwanaga--Gorenstein algebra and stably $3$-Calabi--Yau,
\item $\Lambda$ is Jacobian.
\end{enumerate}
\end{teorema}

\begin{proof}
$(1 \Rightarrow 2)$ Let $\Lambda$ be monomial 2-Calabi--Yau tilted. By  Theorem \ref{prop KR - teo ralf} (1) 
$\Lambda$ is a $1$-Iwanaga--Gorenstein algebra with stably $3$-Calabi--Yau singularity category. 

$(2 \Rightarrow 3)$ By Lemmas \ref{lema 3-CY} and \ref{lema-Jacobian}, a 1-Iwanaga--Gorenstein monomial algebra $\Lambda= kQ/I$ with a $3$-Calabi--Yau singularity category is such that the ideal $I$ arises from the cyclic derivatives of a potential, and this potential can be easily defined from $I$. 

$(3 \Rightarrow 1)$ It follows from \cite[Theorem 3.6]{Am} that $\Lambda$ is 2-Calabi--Yau tilted in the sense of Keller--Reiten.\end{proof}

The previous result is written with the restriction: $b_j \in \mathbb{Z}^+$ is non-zero over $k$. To guarantee this we say $k$ is of characteristic zero.   

To solve the problem presented by integration-differentiation of potentials over fields of positive characteristic, an alternative is proposed: using \emph{hyperpotentials}. A hyperpotential on a quiver $Q$ is a collection of elements $(\rho_\alpha)_{\alpha \in Q_1}$ over the complete algebra $R \langle\langle Q\rangle\rangle$ such that: for $\alpha \colon i \to  j$, $\rho_\alpha$ is a (possibly infinite) linear combination of paths $j \leadsto i$, and $\sum_{\alpha \in Q_1} [\alpha, \rho_\alpha] = 0$. The \emph{Jacobian algebra of a hyperpotential} is the quotient $R \langle\langle Q\rangle\rangle / \overline{ \langle \rho_\alpha\rangle } $ (see \cite[Proposition 1]{Lad}). In view of this, if we admit hyperpotentials, the last result extends to algebraically closed fields of positive characteristic.

\begin{remark} \normalfont Let $\Lambda=kQ/I$ a monomial $k$-algebra, where $k$ is a field of arbitrary characteristic. The following are equivalent: 
\begin{enumerate}
\item $\Lambda$ is $2$-Calabi--Yau tilted,
\item $\Lambda$ is a $1$-Iwanaga--Gorenstein algebra and stably $3$-Calabi--Yau,
\item $\Lambda$ is a Jacobian algebra of a hyperpotential.
\end{enumerate}
\end{remark}

\begin{ejemplo} The algebra $kQ/I$ in Example \ref{ejemplo-1} is Jacobian. The potential is $W= \alpha \beta \gamma \delta + \lambda^5$.
\end{ejemplo}

\subsection{Gentle case}\label{subsec-gentle}

A particular case of monomial algebras is gentle algebras, see Definition \ref{gentle-def}. In particular gentle Jacobian algebras were studied in \cite[Section 2]{ABCP}. There it is proved that a Jacobian algebra arising from a triangulation is gentle and 1-Iwanaga--Gorenstein \cite[Theorem 2.7]{ABCP}, and a gentle algebra $kQ/I$, where $Q$ has no loops, such that relations are restricted to saturated 3-cycles arises from a surface triangulation. 

Now we apply the results in Section 3 but adding the gentle hypothesis.  

\begin{proposicion} Let $kQ/I$ be a gentle 2-Calabi--Yau tilted algebra, then every relation lies on a saturated 3-cycle or a saturated loop.
\end{proposicion}

\begin{proof}
A gentle algebra is monomial, and by Theorem \ref{prop KR - teo ralf} (1) we can apply Lemma \ref{lema 3-CY}. Then we have $r_i = b_i n_i -1 = 2$, so $b_i n_i = 3$. We have only two options: $b_i = 1$ and $n_i = 3$, or $b_i = 3$ and $n_i = 1$. In the first case the relations arise from a potential term of length three $\alpha \beta \gamma$ and we obtain a saturated 3-cycle, and in the second case the potential term is $\delta^3$ so we get a loop such that $\delta^2 \in I$. 
\end{proof}

We can easily build gentle 2-Calabi--Yau tilted algebras. We only need to be careful so we do not create a cycle that is not in $I$.

\begin{ejemplo}\label{ejemploloop} Let $Q$ be the quiver in Figure \ref{ejemploconloop}, and consider the potential $W=\delta_1^3 + \sum_{i=1}^3 \alpha_i \beta_i \gamma_i$. Then $Jac(Q,W)$ is a gentle $2$-Calabi--Yau tilted algebra.

\begin{figure}[h!]

\begin{tikzcd}
&& 7 \arrow[d,"\beta_3"] 
& & 2 \arrow[ll,"\alpha_3"] \arrow[d,"\alpha_1"] && 6 \arrow[drr,"\gamma_2"] \\
8 \arrow[rr,"\lambda_1"] && 4 \arrow[urr,"\gamma_3"]\arrow[rr,"\lambda_2"]
& & 1\arrow[rr,"\beta_1"] & & 3\arrow[ull,"\gamma_1"] \arrow[u,"\beta_2"]& & 5\arrow[ll,"\alpha_2"]\arrow[loop right,"\delta_1"]{r}\end{tikzcd}
\caption{Gentle bound quiver $(Q,I)$, Example \ref{ejemploloop}}
\label{ejemploconloop}
\end{figure}

\end{ejemplo}


\section{Other monomial algebras and their singularity categories}\label{sect-4}

During this section we study $\stabGP$ for a family of gentle algebras and some examples of non-gentle $m$-cluster tilted algebras. We will obtain an analogous property to Theorem \ref{prop KR - teo ralf} (2) by explicit computation. This property is deeply related to the CY dimension of $\stabGP$.\footnote{This discussion is the main topic of the Appendix. The author is deeply grateful with the Appendix authors for their answer on this subject.}

An \emph{algebra arising from an $(m+2)$-angulation} $\Tt$ is defined as $k Q_\Tt / I_\Tt$  where the bound quiver $(Q_\Tt, I_\Tt)$ is obtained from the $(m+2)$-angultion. The set of vertices $(Q_\Tt)_0$ is indexed by the internal arcs in $\Tt$. For any two vertices $i,j$, there is an arrow $i \to j$ when the corresponding $m$-diagonals $i$ and $j$ share a vertex, they are edges of the same $(m+2)$-gon and $i$ follows $j$ clockwise. Given consecutive arrows $i \xrightarrow{\alpha} j \xrightarrow{\beta} k$, then $\alpha \beta \in I_\Tt$ if and only if $i$, $j$ and $k$ are edges in the same $(m+2)$-gon.

\begin{ejemplo}\label{ejemplorevision1}  \begin{figure}[h!]
\centering
\def\svgwidth{4.5in}
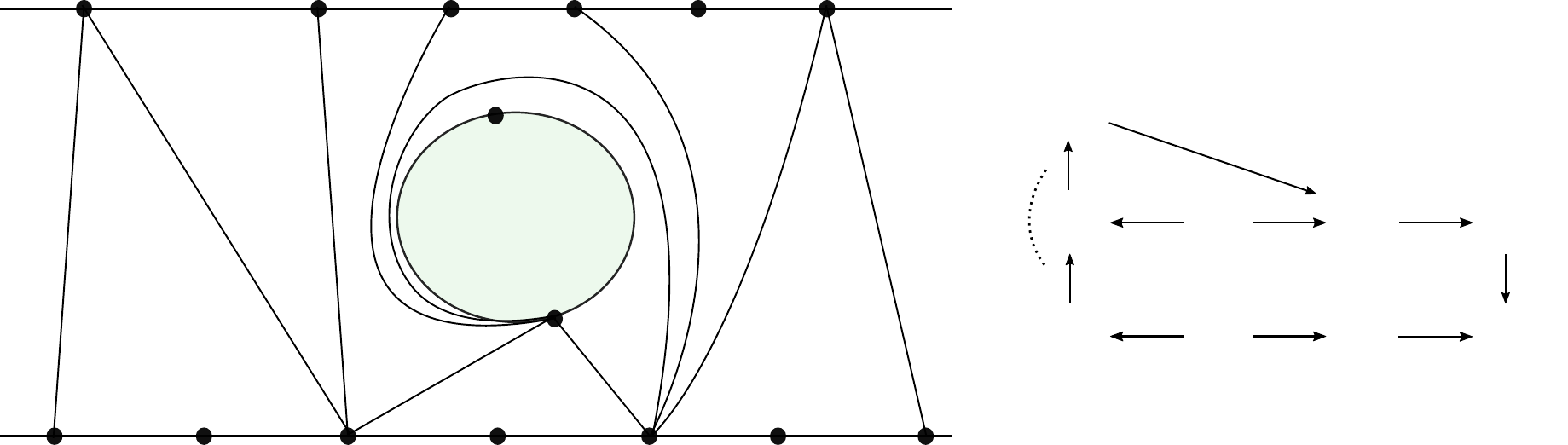
\caption{$4$-angulation of the sphere with 3 three disks removed, and bound quiver $(Q_\Tt, I_\Tt)$.}
\label{m2angulacion}
\end{figure}

Draw an annulus as a strip $\Sigma$ and remove a disk from its interior. Add two marked points in this boundary component and other marked points in the boundary of the strip, see Figure~\ref{m2angulacion}. Given a $4$-angulation of this marked surface, we can obtain a bound quiver $(Q_\Tt, I_\Tt)$.

\end{ejemplo}

In particular, when the surface is a disk (resp. an annulus) the gentle algebra obtained is $m$-cluster tilted of type $\Aa$ (resp. $\tAa$) in the sense of Thomas \cite{T}, as it was studied in several works \cite{CCS,ABCP,Ba,Mu,Tol,Gub}.

The following properties were observed in \cite[Rem. 2.18]{Mu} and \cite[Sec. 7]{Gub} in the case of the disc and annulus, but it is easy to see that they hold when the algebra arises from an $(m+2)$-angulation of a general surface.

\begin{proposicion}\label{m-angulaciones} Let $(Q_\Tt,I_\Tt)$ be a bound quiver arising from an $(m+2)$-angulation.
\begin{enumerate}
\item $\Lambda_\Tt = k Q_\Tt/I_\Tt$ is a gentle algebra.
\item The only possible saturated cycles in $(Q_\Tt,I_\Tt)$ are $(m+2)$-cycles.
\item There can be at most $m - 1$ consecutive zero-relations not lying in a saturated cycle.
\end{enumerate}

\end{proposicion}

Immediately, we have the following observation. 

\begin{lema}\label{lema1} Let $\Lambda_\Tt = k Q_\Tt/I_\Tt$ be an algebra arising from a $(m+2)$-angulation. Then, $\Lambda_\Tt$ is $m$-Iwanaga--Gorenstein.
\end{lema}

\begin{proof}

The case $m=1$ follows from Theorem \ref{prop KR - teo ralf} (1), and also from \cite[Lemma 2.6]{ABCP}. Let $m\geq 2$. Since $\Lambda_\Tt$ is gentle, we can apply Theorem \ref{lema GR}. First assume that there is no gentle arrow in $(Q_\Tt,I_\Tt)$, then $n(\Lambda_\Tt)=0$, so $d$ is zero or one and $d \leq m$. The statement follows.
\\
Now, assume there are gentle arrows in $(Q_\Tt,I_\Tt)$, and let $\alpha_1$ be one of them. It follows that $\alpha_1$ is not part of a saturated cycle. Let $\alpha_1 \ldots \alpha_r$ be a critical path. Since $\alpha_1$ is not part of a saturated cycle, then none of the arrows $\alpha_i$ for $1 \leq i \leq r$ is part of a saturated cycle. By Proposition \ref{m-angulaciones} (3), the maximal number of consecutive zero-relations outside of a saturated cycle is $m-1$. Therefore, $r \leq m$, and by Theorem \ref{lema GR}, $\Lambda_\Tt$ is Gorenstein of dimension $d \leq m$.\end{proof}

Most of the arguments in the following lemma can be found also in \cite[Section 4]{Ka}.

\begin{lema}\label{lemaK} Let $\Lambda =kQ/I$ be a gentle $d$-Iwanaga--Gorenstein algebra, $d \geq 1$. Let $x\in Q_0$, and let $N$ be an indecomposable direct summand of $\rad P(x)$. Then,
\begin{enumerate}
\item[(a)] $N \in \stabGP(\Lambda)$, or
\item[(b)] $\projdim N \leq d-1$.
\end{enumerate}

\end{lema}

\begin{proof} If $N$ is projective, we are in case (b). Let $N$ be non projective. Let  $P(x)=M(u^{-1} \alpha^{-1} \beta w)$ be the indecomposable projective and $N=M(u)$ so that $S(t(\alpha))= \ttop M(u)$. We study the cases:

\begin{enumerate}
\item[(i)] $\alpha $ is part of a saturated cycle $x_1 \ra \cdots \ra x_i \xra{\alpha} x_{i+1} \cdots \ra  x_1$.
\item[(ii)] $\alpha$ is not part of a saturated cycle.
\end{enumerate}
(i) Let $x_i\xra{\alpha} x_{i+1}$, then $M(u)$ is a direct summand of $\rad P(x_{i})$. By Theorem \ref{teorema de kalck}, $M(u)\in \stabGP (\Lambda)$.
\\
(ii) Since $N=M(u)$ is not projective, there exists an arrow $\delta_1 $ such that $\alpha \delta_1 \in I$. The arrow $\delta_1$ is not part of a saturated cycle, since then $\alpha$ would be part of the saturated cycle.
Let $P(t(\alpha))= M(c^{-1} \delta_1^{-1} u)$, then there is an exact sequence
\begin{equation}\label{secdelta1}
0 \ra M(c) \ra P(t(\alpha)) \ra M(u) \ra 0.
\end{equation}
If the string module $M(c)$ is not projective, then it satisfies the same conditions as $M(u)$, so we can construct a new exact sequence

\begin{equation}\label{secdelta2}
0 \ra M(c_1) \ra P(t(\delta_1)) \ra M(c) \ra 0.
\end{equation}
Recursively, we obtain a path $\alpha \delta_1 \cdots \delta_n$ such that each quadratic factor belongs to $I$.
This process has to finish after a finite number of steps, being the direct summand $M(c_n)$ of $P(t(\delta_{n-1}))$ a projective module. If there were not finite steps and $M(c_n)$ was not projective, we would find new arrows $\delta_{n+1}, \ldots$ and form a path $\alpha \delta_1 \cdots \delta_n \cdots$ such that each quadratic factor is in $I$. The quiver $Q$ is finite, so the only way to construct an infinite path $\alpha \delta_1 \cdots \delta_n \cdots$ is by reaching a saturated cycle. By the gentleness, if one of the arrows $\delta_i$ is in a saturated cycle, then all $\alpha,\delta_1, \ldots , \delta_n$ are in the saturated cycle, contradicting the condition imposed on $\alpha$. Therefore the procedure to find the short exact sequences stops. Splicing the short exact sequences we get a projective resolution for $M(u)$, which is finite, so $\projdim M(u) < \infty$. By Remark \ref{rema gorenstein}, we have $\projdim M(u) \leq d$. Now, we can also express $M(u)$ as $M(u)=\Omega M(\beta w)$. If we had $\projdim M(u)=d$, then we would have $\projdim M(\beta w)=d+1$ and this is impossible by Remark \ref{rema gorenstein}. Thus, $\projdim M(u)\leq d-1$. \end{proof}

To complete the previous lemma, observe that if $\Lambda$ is selfinjective (that is, $\Lambda$ is $0$-Iwanaga--Gorenstein) then every indecomposable module is Gorenstein-projective.

\begin{teorema}\label{main} Let $\Lambda_\Tt = k Q_\Tt /I_\Tt$ be an algebra arising from a $(m+2)$-angulation and let $N$ be a $\Lambda_\Tt$-module. Then $N$ is Gorenstein projective if and only if $\Omega^{m+1}\tau N \simeq N$ in $\stabfgmod \Lambda_\Tt$. In particular, this holds for $m$-cluster tilted algebras of type $\Aa$ and $\tAa$.
\end{teorema}

\begin{proof}
Let $M$ be an indecomposable module in $\stabGP(\Lambda_\Tt)$. By Theorem \ref{teorema de kalck}, $M=M(u_i)$ where $u_i$ is the maximal non-zero path starting at $x_i$.

\begin{center}

\begin{tikzcd}[column sep=small]
\ \arrow[d,rightsquigarrow,"v_{i}"]& & \  \arrow[d,rightsquigarrow,"v_{i+1}"] & & \ \arrow[d,rightsquigarrow,"v_{i+2}"] \\ x_i \arrow[rr,"\alpha_i"] \arrow[d,rightsquigarrow,"u_i"] & & x_{i+1} \arrow[rr,"\alpha_{i+1}"] \arrow[d,rightsquigarrow,"u_{i+1}"]& & x_{i+2}\arrow[d,rightsquigarrow,"u_{i+2}"] \\
\ & & \ & & \
\end{tikzcd}
\end{center}
We compute a minimal projective presentation of $M(u_i)$.
\[\xymatrix@C8pt@R8pt{&M( u_{i+1}^{-1} \alpha_{i+1} u_{i+2}) \ar[rr]^{p_1}\ar@{->>}[rd] && M( u_i^{-1} \alpha_i u_{i+1})
\ar[rr]&& M(u_i) \ar[rr]&&0.\\ M(u_{i+2}) \ar@{^{(}->}[ru]&&M(u_{i+1}) \ar@{^{(}->}[ru]}\]
Observe that $\Omega^t M(u_i) = M( u_{i+t})$, where $t$ is an integer considered modulo $m+2$. Applying the Nakayama functor we get

\begin{displaymath}
	\xymatrix  @R=0.6cm  @C=0.6cm {
		0\ar[r] & M(v_{i+1})\ar@{^(->}[r] & M(v_i \alpha_i  v_{i+1}^{-1} )\ar[rr]^{\nu p_1}\ar@{^(->}[rd] && M(  v_{i-1}\alpha_{i-1} v_i^{-1} )\ar@{>>}[r]&M(v_{i-1}).\\
		&&& M(v_i)\ar@{>>}[ru] &&}
\end{displaymath}
Then $ \tau M(u_i) = \ker \nu p_1 = M(v_{i+1})$. Let $S(y_{i+1})= \ttop M(v_{i+1})$. Let $P(y_{i+1})= M(w^{-1}_{i+1} v_{i+1} \alpha_{i+1} u_{i+2})$ be the projective cover of $M(v_{i+1})$. 

\begin{center}

\begin{tikzcd}[column sep=small]
& & y_{i+1}  \arrow[d,rightsquigarrow,"v_{i+1}"]\arrow[rr,rightsquigarrow,"w_{i+1}"] & & \ \\
x_i \arrow[rr,"\alpha_i"] & & x_{i+1} \arrow[rr,"\alpha_{i+1}"] & & x_{i+2}\arrow[d,rightsquigarrow,"u_{i+2}"] \\
& & & & \
\end{tikzcd}
\end{center}
Therefore, $\Omega M(v_{i+1}) = M(w'_{i+1}) \oplus M(u_{i+2})$, where $M(w'_{i+1})$ is the maximal submodule of $M(w_ {i+1})$. The syzygy functor is additive, so
\begin{equation*}
\Omega^{m+1} \tau M(u_i) = \Omega^{m+1} M (v_{i+1}) = \Omega^m M(w'_{i+1}) \oplus \Omega^m M(u_{i+2}).
\end{equation*}
Since $\Omega^t M(u_i) = M( u_{i+t}) $, we have
\begin{equation*}
\Omega^m M(w'_{i+1}) \oplus \Omega^m M(u_{i+2}) = \Omega^m M(w'_{i+1}) \oplus M(u_i).
\end{equation*}
Now, we only need to prove that $\Omega^m M(w'_{i+1})=0$. Observe that $ M(w'_{i+1})$ is a direct summand of $\rad P(y_{i+1})$. 

\smallskip

We know by Lemma \ref{lema1} that $\Lambda_\Tt$ is Gorenstein of dimension $d \leq m$. By Lemma \ref{lemaK} one of the following holds: 
\begin{enumerate}
\item $\projdim M(w'_{i+1}) \leq m-1$, or
\item $M(w'_{i+1}) \in \stabGP(\Lambda_\Tt)$. 
\end{enumerate}

If (1) holds, then $\Omega^m M(w'_{i+1})=0$ and we are done.

\smallskip

We assume (2) holds, so $M(w'_{i+1}) \in \stabGP(\Lambda_\Tt)$. We prove that this leads to a contradiction.  Let $z_{i+1}$ be the vertex such that $ \ttop M( w'_{i+1}) =S (z_{i+1})$. By the description in Theorem \ref{teorema de kalck}, the vertex $z_{i+1}$ is the target of an arrow $\gamma$ in a saturated $(m+2)$-cycle and $\gamma w'_{i+1}\neq 0$. 

\smallskip

(2a) If the arrow $\gamma$ is $ y_{i+1} \xra{\gamma} z_{i+1}$, as appears in the figure below, then there is an arrow $a_j$ in the saturated $(m+2)$-cycle, such that  $a_j \gamma \in I_{\Tt}$. Then, $a_j v_{i+1} \neq 0$, contradicting that $I(x_{i+1}) = M( v_i \alpha_i  v_{i+1}^{-1} )$ is the indecomposable injective associated to $x_i$. This is absurd.

\begin{center}

\begin{tikzcd}[column sep=small]
\ \arrow[d,"a_{j}"]& & \   & & \  \\ y_{i+1} \arrow[rr,"\gamma"] \arrow[d,rightsquigarrow,"v_{i+1}"] & & z_{i+1} \arrow[d,rightsquigarrow,"w'_{i+1}"] \arrow[rr,"a_{j+2}"]& & \  \\
\ & & \ & & \
\end{tikzcd}
\end{center}

(2b) If the arrow $\gamma$ in a saturated cycle is such that $s(\gamma)\neq y_{i+1}$, as we see in the figure below, there is an arrow $b_{j+2}$ following the saturated cycle such that $\gamma b_{j+2} \in I_\Tt$. Thus, we have $\gamma w'_{i+1} \neq 0$ and by gentleness, $a_ {j+1}b_{j+2}\notin I_\Tt$.  

\begin{center}

\begin{tikzcd}[column sep=small]
\ \arrow[d,dashrightarrow,"a_j"] & & \ \arrow[d,"\gamma"]  & & \  \\ y_{i+1} \arrow[rr,"a_{j+1}"] \arrow[d,rightsquigarrow,"v_{i+1}"] & & z_{i+1} \arrow[rr,"b_{j+2}"] \arrow[d,rightsquigarrow,"w'_{i+1}"] & & \  \\
\ & & \ & & \
\end{tikzcd}
\end{center}

But recall that $M(w'_{i+1})$ is a submodule of $\rad P(y_{i+1})$, so $b_{j+2}$ has to be the first arrow in the string $w'_{i+1}$. This contradicts that $M(w'_{i+1})$ is a submodule of $\rad P(y_{i+1})$. To sum up, $M(w'_{i+1})$ is not a trivial Gorenstein projective module and just $(1)$ holds. \end{proof}


As a corollary, we obtain the next result that generalizes the properties known for cluster-tilted algebras: Theorem \ref{prop KR - teo ralf} (1) and (2).

\begin{corolario}\label{teo m tipo a} Let $\Lambda$ be a $m$-cluster tilted algebra of type $\Aa$ or $\tAa$. Then
\begin{enumerate}
\item $\Lambda$ is $m$-Iwanaga--Gorenstein, and
\item $N \in \GP(\Lambda)$ if and only if $\Omega^{m+1} \tau N \cong N$.
\end{enumerate}
\end{corolario}

\begin{proof}
Part (1) follows from Lemma \ref{lema1}. Part (2) follows from Theorem \ref{main}. 
\end{proof}

From a more general point of view, $m$-cluster tilted algebras are $(m+1)$-Calabi--Yau tilted algebras. It is known, and expected, that Corollary \ref{teo m tipo a} does not hold in general for $d$-Calabi--Yau tilted algebras by the next reasons:

\begin{enumerate}
\item In \cite[Section 5.3]{KR} there is an example (due to Iyama) of a $d$-Calabi--Yau tilted algebra that \emph{is not} Iwanaga--Gorenstein.

\item Moreover, a recent preprint \cite{Lad2} shows that all finite dimensional $k$-algebras are $d$-Calabi--Yau tilted (for all $d > 2$).
\end{enumerate}

Still, there are results in this direction due to Keller and Reiten \cite[Section 4.6]{KR2}, and Beligiannis \cite[Theorem 6.4]{Be2}, showing that a $d$-Calabi--Yau tilted algebra is Iwanaga--Gorenstein under certain conditions. In both cases, it is required that $\add T$ is corigid to some degree $u$, meaning that $\Hom_\Cc(T,T[-t])=0$ for all $1\leq t\leq u$. Over $m$-cluster tilting categories of type $\Aa$ or $\tAa$, the full subcategory defined by a cluster tilting object $\add T$ is $(m+1)$-cluster tilting. In the next example we show that Corollary \ref{teo m tipo a} is independent of these results, by giving an example of a non-corigid cluster-tilting object in the $2$-cluster category of type $\Aa_4$.

\begin{ejemplo}\label{ejemploA} Let $\Cc^2_Q$ be the $2$-cluster category, where $Q$ is of type $\Aa_4$ and let $T$ be as in Figure \ref{A2}. The object $T$ is not corigid since $\Hom(T_3,T_1[-1]) \simeq \Hom (T_3[1],T_1) \neq 0$. By Corollary \ref{teo m tipo a} the $2$-cluster tilted algebra $\End(T)$ is $2$-Iwanaga--Gorenstein. In fact, in this example the algebra, given by the quiver below bound by $\beta \alpha=0$, is of global dimension two.

\begin{figure}[h]
\centering
\def\svgwidth{4in}
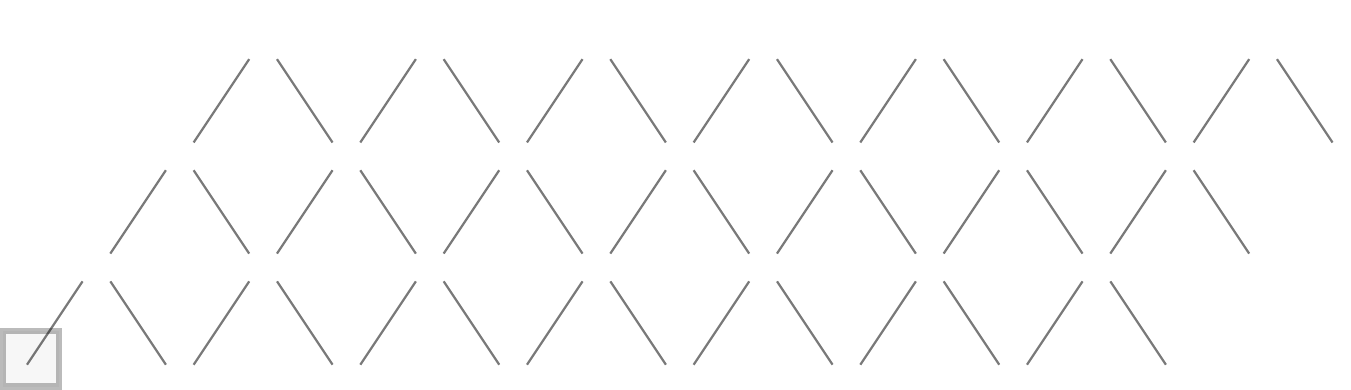
\caption{A $2$-cluster tilting object $T$ which is not corigid.}
\label{A2}
\end{figure}

\begin{figure}[h!]

\begin{tikzcd}
1 & 2\arrow[l,"\alpha"] & 3 \arrow[l,"\beta"] & 4\arrow[l]
\end{tikzcd} 
\end{figure}

\end{ejemplo}

In the next example we see that there are $m$-cluster-tilted algebras not arising from $(m+2)$-angulations, defined by non-corigid cluster tilting objects, for which the conclusion of Corollary \ref{teo m tipo a} holds.

\begin{ejemplo}\label{ejemploD} Let $\Cc^2_Q$ be the $2$-cluster category of type $\mathbb{D}_6$, and $T=\bigoplus_{i=1}^6 T_i$ the $2$-cluster tilting object in Figure \ref{D2}. The algebra $\Lambda=\End_{\Cc^2_Q}(T)$ in Example \ref{ejemploD} is given by the quiver in Figure \ref{Q ejemplo} and the ideal $I= \langle \lambda \alpha, \alpha \beta \gamma, \beta \gamma \delta, \delta \lambda  \rangle$.  We see that $\Lambda$ is $2$-Iwanaga--Gorenstein and has infinite global dimension. The indecomposable modules in $\stabGP(\Lambda)$ are $3,6,\begin{smallmatrix}5\\4\end{smallmatrix},\begin{smallmatrix}2\\1\end{smallmatrix}$, exactly those such that $\Omega^{3} \tau N = N$.

\begin{figure}[h]
\centering
\def\svgwidth{5.8in}
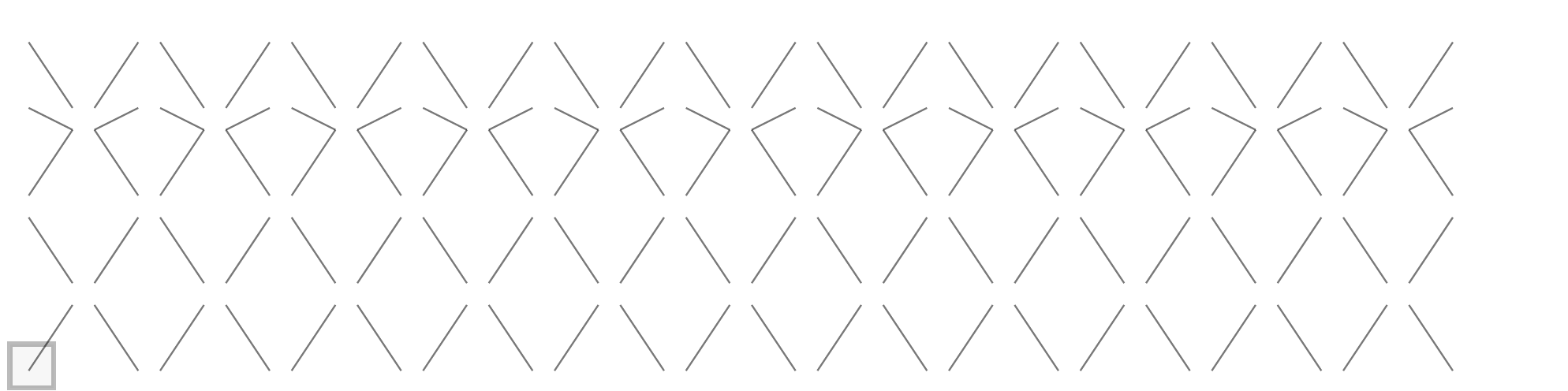
\caption{A $2$-cluster tilting object $T$ such that $\Hom (T_5, T_3 [-1])=\Hom(T_5[1],T_3) \neq 0$, so $T$ is not corigid.}
\label{D2}
\end{figure}
\end{ejemplo}

\begin{figure}[h!]

\begin{tikzcd}[column sep=.2in, row sep=.18in]
1 && 2 \arrow[ll,"\epsilon"] \arrow[d,"\lambda"] && 3 \arrow[ll,"\delta"] && 4 \arrow[ll, "\gamma"] \\
&& 6 \arrow[rr,"\alpha"] && 5 \arrow[urr,"\beta"']
\end{tikzcd}
\caption{ $\Lambda = kQ/I$, Example \ref{ejemploD}}
\label{Q ejemplo}
\end{figure}

As a final remark, we may ask the next question:

\begin{question} $m$-cluster tilted algebras are still a special subclass of $(m+1)$-Calabi--Yau tilted algebras. Are they always Iwanaga--Gorenstein? Does Corollary \ref{teo m tipo a} always hold for these algebras?
\end{question}

\textbf{Acknowledgments:} This work was supported by the Austrian Science Fund Project Number P25647-
N26. The author wants to thank the researchers at Universit{\"a}t Stuttgart,  especially Prof. Steffen Koenig and Dr. Matthew Pressland, and also would like to thank the organizers of BIREP-2017 Summer School on Gentle Algebras.

{}

\bigskip{\footnotesize%

  \textsc{Institut f{\"u}r Mathematik und
Wissenschaftliches Rechnen, Universit{\"a}t Graz, Heinrichstra{\ss}e 36, A-8010 Graz, Austria} \par
  \textit{E-mail address:} \texttt{ana.garcia-elsener@uni-graz.at} 
}

\clearpage
\appendix
\section{Auslander--Reiten translations for Gorenstein algebras\\\footnotesize{by Sondre Kvamme and Matthew Pressland}}

Let $\Lambda$ be an Iwanaga--Gorenstein algebra. Since $\GP(\Lambda)$ is a functorially finite subcategory of $\fgmod{\Lambda}$, it has Auslander--Reiten sequences \cite{A}, inducing an Auslander--Reiten translation $\tau_{\GP}$ on $\stabGP(\Lambda)$, typically different from the Auslander--Reiten translation $\tau_\Lambda$ on $\stabfgmod{\Lambda}$. The goal of this appendix is to relate the objects $\tau_\Lambda M$ and $\tau_{\GP}M$ of $\stabfgmod{\Lambda}$ when $M$ is Gorenstein projective. Indeed, we will show that these objects eventually coincide after repeated application of the syzygy functor. While this fact may not be surprising to experts, and can be deduced quickly from results already in the literature (most easily from \cite[Thm.~3.7]{A1}), we give here a very direct proof, which even exhibits a natural isomorphism of functors. Moreover, we explain how this result provides a new perspective on results of Garcia Elsener (Corollary~\ref{teo m tipo a} of the present paper) and Garcia Elsener--Schiffler \cite[Thm.~1]{GES} characterising Gorenstein projective modules over certain Calabi--Yau tilted algebras $\Lambda$, by relating these characterisations directly to a Calabi--Yau property of $\stabGP(\Lambda)$.

We denote by $\Omega\colon\stabfgmod{\Lambda}\to\stabfgmod{\Lambda}$ the syzygy functor, taking a module to the kernel of a projective cover, and by $\Sigma\colon\stabfgmod{\Lambda}\to\stabfgmod{\Lambda}$ its left adjoint, taking a module to the cokernel of a right $\proj{\Lambda}$-approximation. The restrictions of these functors to $\stabGP(\Lambda)$ are mutually inverse, the restriction of $\Sigma$ being the suspension functor on this triangulated category.

\begin{lem}
\label{sigma-to-GP}
The functor $\Sigma^d\colon\stabfgmod{\Lambda}\to\stabfgmod{\Lambda}$ has essential image in $\stabGP(\Lambda)$.
\end{lem}
\begin{proof}
Since $\Lambda$ is $d$-Iwanaga--Gorenstein, we have $\Omega^d\Sigma^dM\in\stabGP(\Lambda)$ for any $\Lambda$-module $M$. Then $\Sigma^d\Omega^d\Sigma^dM\in\stabGP(\Lambda)$ since $\Sigma$ preserves Gorenstein projectivity. It follows from the triangular identities for the unit and counit of the adjoint pair $(\Sigma^d,\Omega^d)$ that $\Sigma^dM$ is a summand of $\Sigma^d\Omega^d\Sigma^dM$, and hence is itself Gorenstein projective.
\end{proof}

\begin{thm}
\label{mainthm}
Let $\Lambda$ be a finite-dimensional Iwanaga--Gorenstein algebra of Gorenstein dimension at most $d$. Then there is a natural isomorphism
\[\Omega^d\tau_\Lambda\simeq\Omega^d\tau_{\GP}\]
of endofunctors of $\stabGP(\Lambda)$.
\end{thm}
\begin{proof}
Using the Yoneda embedding, it suffices to show that there is a natural isomorphism
\[\stabHom_\Lambda(-,\Omega^d\tau_{\GP}X)\iso\stabHom_\Lambda(-,\Omega^d\tau_\Lambda X)\]
of functors on $\stabfgmod{\Lambda}$, for any $X\in\GP(\Lambda)$.
By adjunction and the Auslander--Reiten formula for $\GP(\Lambda)$, we obtain natural isomorphisms
\[\stabHom_\Lambda(-,\Omega^d\tau_{\GP}X)\cong\stabHom_\Lambda(\Sigma^d-,\tau_{\GP}X)\cong\kdual\Ext^1_\Lambda(X,\Sigma^d-),\]
The validity of the second isomorphism depends on Lemma~\ref{sigma-to-GP}, showing that $\Sigma^d$ takes values in $\stabGP(\Lambda)$. Alternatively, using the Auslander--Reiten formula in $\fgmod{\Lambda}$, we get
\[\stabHom_\Lambda(-,\Omega^d\tau_\Lambda X)\cong\stabHom_\Lambda(\Sigma^d-,\tau_\Lambda X)\cong\kdual\Ext^1_\Lambda(X,\Sigma^d-).\]
The two right-hand sides coincide, completing the proof.
\end{proof}

\begin{thm}[{\cite[Thm.~1]{GES}}]
\label{ge1}
Let $\Lambda$ be a $2$-Calabi--Yau tilted algebra. Then a $\Lambda$-module $M$ is Gorenstein projective if and only if $M\iso\Omega^2\tau_\Lambda M$ in $\stabfgmod{\Lambda}$.
\end{thm}

\begin{proof}
By a result of Keller and Reiten \cite[Prop.~2.1]{KR}, a $2$-Calabi--Yau tilted algebra is Iwanaga--Gorenstein of Gorenstein dimension at most $1$. Thus $\Omega^2 N\in\stabGP(\Lambda)$ for any $N\in\fgmod{\Lambda}$, and the `if' part of the theorem follows immediately.

For the `only if' direction, Keller and Reiten also prove \cite[Thm.~3.3]{KR} that $\stabGP(\Lambda)$ is a $3$-Calabi--Yau category, meaning that $\Sigma^3$ is a Serre functor. On the other hand, the Auslander--Reiten formula states that $\Sigma\tau_{\GP}$ is always a Serre functor on $\stabGP(\Lambda)$. Since Serre functors are unique up to natural isomorphism, the $3$-Calabi--Yau property is equivalent to there being a natural isomorphism $\tau_{\GP}\simeq\Sigma^2$. Now by Theorem~\ref{mainthm}, using again the Iwanaga--Gorenstein property of $\Lambda$, we have
\[\Omega^2\tau_\Lambda M\cong\Omega^2\tau_{\GP}M\cong\Omega^2\Sigma^2 M\cong M.\qedhere\]
\end{proof}

We close by observing that characterisations of Gorenstein projective $\Lambda$-modules as in Theorem~\ref{ge1} are closely related to Calabi--Yau properties of $\stabGP(\Lambda)$ in wider generality.

\begin{prop}
\label{char-to-CY}
Let $\Lambda$ be Iwanaga--Gorenstein of Gorenstein dimension at most $m$. Then $\stabGP(\Lambda)$ is $(m+1)$-Calabi--Yau if and only if there is a natural isomorphism $\Omega^m\tau_\Lambda\simeq\id_{\stabGP(\Lambda)}$. In this case, a $\Lambda$-module $M$ is Gorenstein projective if and only if $M\cong\Omega^m\tau_\Lambda M$ in $\stabfgmod{\Lambda}$.
\end{prop}
\begin{proof}
As in the proof of Theorem~\ref{mainthm}, $\stabGP(\Lambda)$ is $(m+1)$-Calabi--Yau if and only if there is a natural isomorphism $\tau_{\GP}\simeq\Sigma^m$ on $\stabGP(\Lambda)$. Since $\Sigma^m$ has quasi-inverse $\Omega^m$ on this category, this is equivalent to asking that $\Omega^m\tau_{\GP}\simeq\id_{\stabGP(\Lambda)}$. But Theorem~\ref{mainthm} tells us that $\Omega^m\tau_{\GP}\simeq\Omega^m\tau_\Lambda$ in this setting, and the first part of the statement follows. The second part is then proved as in Theorem~\ref{ge1}, replacing $2$ by $m$. 
\end{proof}

Note that the isomorphisms constructed in the proof of Theorem~\ref{ge1} arise from a natural isomorphism in this way, using the $3$-Calabi--Yau property of $\stabGP(\Lambda)$ \cite[Thm.~3.3]{KR}.
We also deduce that a similar Calabi--Yau property holds for certain gentle $m$-cluster-tilted algebras.

\begin{cor}
Let $\Lambda$ be an $m$-cluster-tilted algebra of type $\Aa$ or $\tAa$. Then $\stabGP(\Lambda)$ is $(m+1)$-Calabi--Yau.
\end{cor}

\begin{proof}
By Corollary~\ref{teo m tipo a}, $\Lambda$ is Iwanaga--Gorenstein of Gorenstein dimension at most $m$, and $M\cong\Omega^m\tau_\Lambda M$ for each $M\in\GP(\Lambda)$. Since $\Lambda$ is a gentle algebra, meaning that $\stabGP(\Lambda)$ is semi-simple \cite{Ka}, there must be a collection of such isomorphisms that is natural in $M$---for example, by first choosing such an isomorphism for each $M$ in a set of representatives for the isoclasses of indecomposable objects, and then extending to arbitrary objects by choosing direct sum decompositions. (We use here that the ground field $k$ is algebraically closed, so that each indecomposable has endomorphism ring $k$.) The result then follows by Proposition~\ref{char-to-CY}.
\end{proof}

\textbf{Acknowledgments:} This appendix originated from discussions between Matthew Pressland and Ana Garcia Elsener after the latter's talk at the BIREP Summer School on Gentle Algebras in August 2017. We thank the organisers of the school for the inspiring environment.

{}

\bigskip{\footnotesize%

  \textsc{SK: Laboratoire de Math\'ematiques d'Orsay, Universit\'e Paris-Sud, 91405 Orsay, France} \par
  \textit{E-mail address:} \texttt{sondre.kvamme@u-psud.fr} \par
  \addvspace{\medskipamount}
  
  \textsc{MP: Institut f\"ur Algebra und Zahlentheorie, Universit\"at Stuttgart, Pfaffenwaldring 57, 70569 Stuttgart, Germany} \par
  \textit{E-mail address:} \texttt{presslmw@mathematik.uni-stuttgart.de}


\end{document}